\documentclass[12pt,reqno]{amsart}

\usepackage[a4paper,margin=1in]{geometry}
\usepackage{amsmath,amssymb,amsthm,mathtools}
\usepackage{bbm}
\usepackage{enumitem}
\usepackage{hyperref}
\usepackage[normalem]{ulem}

\usepackage{pdflscape}
\usepackage{subcaption}
\usepackage{caption}
\usepackage{tikz, float, tikzscale}
\usepackage{pgfplots}
\pgfplotsset{compat=1.18}

\usepackage[T1]{fontenc}
\usepackage{libertinus,libertinust1math}

\theoremstyle{definition}

\newtheorem{assumption}{Assumption}
\theoremstyle{plain}
\newtheorem{thm}{Theorem}
\newtheorem*{thm*}{Theorem}

\newtheorem{lemma}{Lemma}

\title{Quantitative longest-run laws for partial quotients}
\author{Ying Wai Lee}

\begin{document}

\begin{abstract}
    Two longest-run statistics are studied: the longest run of a fixed value and the longest run over all values. Under quantitative mixing and exponential cylinder estimates for constant words, a general theorem is proved. Quantitative almost-sure logarithmic growth is obtained, and eventual two-sided bounds with double-logarithmic error terms are established. For continued-fraction partial quotients, explicit centring constants and double-logarithmic error bounds are derived for both statistics.
\end{abstract}

\maketitle

\section{Introduction}\label{sec:introduction}

Longest-run statistics form a fundamental class of observables at the interface of probability, dynamical systems, and metric number theory. For symbolic processes taking values in a countable alphabet, a natural question is, what is the length of the longest consecutive run of a single symbol?

In probability, for independent sequences, classical Erd\H{o}s--R\'enyi and Erd\H{o}s--R\'ev\'esz laws describe the growth of longest runs; analogous statements persist for weakly dependent sequences under suitable mixing assumptions \cite{grill1987erdHos}. In dynamical systems, longest-run statistics arise as non-additive observables associated with symbolic codings of orbits. In number-theoretic expansions, they quantify the persistence of digits and capture a notion of persistence that is complementary to magnitude-based information.

The attention of the present work is turned to longest-run behaviour for partial quotients, both for a fixed prescribed value and for the maximised run over all values. A general theorem is established for shift-invariant symbolic processes under verifiable mixing and cylinder estimates, yielding quantitative almost-sure longest-run asymptotics with an explicit convergence rate. The required hypotheses are then verified for the Gauss system, and corresponding longest-run results for partial quotients are obtained.

\section{Continued fractions}

Classical work on continued fractions has emphasised magnitude-related phenomena for partial quotients (such as extremes and growth rates), and more recently the range-renewal (distinct-value count) statistic has also been investigated \cite{WuXie2017RangeRenewalCF}. 

For any $x\in[0,1)\setminus\mathbb{Q}$, there exists a unique sequence of positive integers $(a_n)_{n\in\mathbb{N}}$ such that $x$ admits the continued fraction expansion:
\begin{align*}
    x=[a_1,a_2,a_3,\dots]\coloneqq\frac{1}{\displaystyle a_1+\frac{1}{\displaystyle a_2+\frac{1}{a_3+\cdots}}},
\end{align*}
where for any $n\in\mathbb{N}$, $a_n\coloneqq a_n(x)\in\mathbb{N}$ is referred to as $n$-th partial quotient of $x$. 
Define, for any $n\in\mathbb{N}$, the fixed-symbol longest-run function $L_n:([0,1)\setminus\mathbb{Q})\times\mathbb{N}\to\mathbb{N}\cup\{0\}$ by, for any $x\in [0,1)\setminus\mathbb{Q}$ and $\lambda\in\mathbb{N}$,
\begin{align}\label{eq:Ln-def}
    L_n(x,\lambda)
    \coloneqq\max_{0\leq j\leq n-1}\sum_{k=1}^{n-j}\prod_{i=1}^k\mathbf{1}_{a_{j+i}(x)=\lambda},
\end{align}
and the maximised longest-run function $R_n:[0,1)\setminus\mathbb{Q}\to\mathbb{N}$ by, for any $x\in[0,1)\setminus\mathbb{Q}$,
\begin{align}\label{eq:Rn-def}
    R_n(x)\coloneqq \max_{\lambda\in\mathbb N} L_n(x,\lambda).
\end{align}
In other words, for any $x\in[0,1)\setminus\mathbb{Q}$ and $\lambda,n\in\mathbb{N}$, $L_n(x,\lambda)$ is the maximum length of a consecutive run of the value $\lambda$ within the first $n$ partial quotients of $x$; and $R_n(x)$ is the maximum length of a consecutive run over all values. For example, one obtains $L_{30}(\pi-3,1)=3$ and $L_{30}(\pi-3,2)=R_{30}(\pi-3)=4$ by observing the first 30 partial quotients of $\pi-3$:
\begin{align*}
    \pi-3=[7,15,1,292,
    \dotuline{1},\dotuline{1},\dotuline{1},
    2,1,3,1,14,2,1,1,
    \dotuline{2},\dotuline{2},\dotuline{2},\dotuline{2},
    1,84,2,1,1,15,3,13,1,4,2,\ldots]
\end{align*}
where the underlined blocks show the longest runs.

These statistics are dynamical analogues of longest-run functionals for independent and identically distributed sequences. The key difference is that the process of partial quotients is not independent; nevertheless it is shift-invariant and strongly mixing under the Gauss measure, so that the setting falls within the broader mixing longest-run framework \cite{grill1987erdHos}. Two additional difficulties are imposed by the countable alphabet and by the maximisation over all digit values in \eqref{eq:Rn-def}.

Song and Zhou~\cite[Theorem~1.1]{SongZhou2020LongestBlock} established the first-order logarithmic law for the fixed-symbol longest-run function: for any $\lambda\in\mathbb{N}$ and almost every $x\in[0,1)\setminus\mathbb{Q}$,
\begin{align}\label{eq:first-order-L}
    \lim_{n\to\infty}\frac{L_n(x,\lambda)}{\log{n}/\log{\tau}}=\frac12,
    \qquad
    \tau\coloneqq\tau(\lambda)\coloneqq \frac{\lambda+\sqrt{\lambda^2+4}}{2}.
\end{align}
Wang and Wu~\cite{WangWu2011MaxRunLength} established the first-order behaviour of the maximised longest-run function: for almost every $x\in[0,1)\setminus\mathbb{Q}$,
\begin{align}\label{eq:first-order-R}
    \lim_{n\to\infty}\frac{R_n(x)}{\log{n}/\log{\varphi}}=\frac12,
    \qquad
    \varphi\coloneqq\frac{1+\sqrt5}{2}.
\end{align}
Song and Zhou also studied Hausdorff dimensions of the associated level and exceptional sets; some other related dimension results were obtained by Tan and Zhou \cite{TanZhou2025UniformDiophantineRunLength}.

The law in \eqref{eq:first-order-L} indicates that the maximised statistic in \eqref{eq:first-order-R} is governed by the most probable value, namely the positive integer 1 as partial quotients. The limits~\eqref{eq:first-order-R} and~\eqref{eq:first-order-L} identify the leading logarithmic scale of the longest-run statistics. However, corresponding convergence rates have not been quantified. An eventual two-sided additive refinement is established, sharpening both \eqref{eq:first-order-R} and \eqref{eq:first-order-L} by providing an explicit convergence rate.

\section{Main results}

A general theorem is first stated in order to isolate the probabilistic mechanism behind the eventual two-sided additive error bounds for longest runs. Quantitative mixing permits well-separated blocks to be treated as nearly independent, while exponential estimates for constant-word events calibrate the typical run length. Under these inputs, the longest-run length is shown to concentrate with a double-logarithmic additive error.

Let $(\Omega,\mathcal F,\mu,T)$ be a measure-preserving dynamical system. Let $\mathcal{A}$ be a countable set and $X:\Omega\to\mathcal{A}$ be a $\mathcal{F}$-measurable function. Let $(X_k)_{k\in\mathbb{N}}$ be the dynamical process generated by $X$; that is, for any $k\in\mathbb{N}$, $X_k:\Omega\to\mathcal{A}$ is a $\mathcal{F}$-measurable function and for any $\omega\in\Omega$, 
\begin{align*}
    X_k(\omega)\coloneqq\left(X\circ T^{k-1}\right)(\omega).
\end{align*}
Define, for any $n\in\mathbb{N}$, the fixed-symbol longest-run statistic $L_n:\Omega\times\mathcal{A}\to\mathbb{N}\cup\{0\}$ by, for any $\omega\in\Omega$ and $m\in\mathcal{A}$,
\begin{align*}
    L_n(\omega,m)
    \coloneqq\max_{0\leq j\leq n-1}\sum_{k=1}^{n-j}\prod_{i=1}^k\mathbf{1}_{X_{j+i}(\omega)=m};
\end{align*}
and the maximised longest-run statistic $R_n:\Omega\to\mathbb{N}$ by, for any $\omega\in\Omega$,  
\begin{align*}
    R_n(\omega)\coloneqq \max_{m\in\mathcal{A}} L_n(\omega,m).
\end{align*}
In other words, for any $\omega\in\Omega$ and $m\in\mathcal{A}$ and $n\in\mathbb{N}$, $L_n(\omega,m)$ is the maximum length of a consecutive run of the symbol $m$ within $X_1(\omega),\dots,X_n(\omega)$, and $R_n(\omega)$ is the maximum length of a consecutive run over all symbols. 

Three assumptions are imposed. Assumption~\ref{ass: 1} is a quantitative mixing condition with a gap. Assumption~\ref{ass: 2} fixes the exponential scale of constant-word events for a prescribed symbol. Assumption~\ref{ass: 3} supplies the corresponding uniform summed bound needed for the maximised statistic.

\begin{assumption}\label{ass: 1}
There exist constants $C_0>0$ and $\theta\in(0,1)$ such that for all $r,g\ge1$, all
events $A\in\sigma(X_1,\dots,X_r)$ and
$B\in\sigma(X_{r+g+1},X_{r+g+2},\dots)$,
\begin{align}\label{eq: ass 1}
    \left|\mu(A\cap B)-\mu(A)\mu(B)\right|
    \leq C_0\,\theta^{g}\,\mu(A)\mu(B).
\end{align}
\end{assumption}

Define, for any $m\in\mathcal{A}$ and $k\in\mathbb{N}$, the constant-word cylinder event:
\begin{align*}
    \Delta_k{\left(m\right)}
    \coloneqq \bigcap_{j=1}^k\left\{\omega\in\Omega: X_j(\omega)=m\right\}.
\end{align*}
Let $\rho>1$.

\begin{assumption}\label{ass: 2}
There exist $m_*\in\mathcal{A}$ and $c_-,c_+>0$ such that for any $k\in\mathbb{N}$,
\begin{align}\label{eq: ass 2}
    c_-\,\rho^{-2k}
    \leq \mu{\left(\Delta_k{\left({m_*}\right)}\right)}
    \leq c_+\,\rho^{-2k}.
\end{align}
\end{assumption}
\begin{assumption}\label{ass: 3}
There exists $C_1>0$ such that for any $k\in\mathbb{N}$,
\begin{align}\label{eq: ass 3}
    \sum_{m\in\mathcal{A}} \mu{\left(\Delta_k{\left({m}\right)}\right)}
    \leq C_1\,\rho^{-2k}.
\end{align}
\end{assumption}

\begin{thm}\label{thm: 1}
Suppose Assumption~\ref{ass: 1} holds.
\begin{enumerate}[leftmargin=*]
\item 
Suppose Assumption~\ref{ass: 2} holds.
Then for any $c>1/2$ and $\mu$-almost every $\omega\in\Omega$, there exists $n_{c,\omega}\in\mathbb{N}$ such that for any $n\in\mathbb{N}$, if $n\geq n_{c,\omega}$ then:
\begin{align*}
    \left|L_n(\omega,m_*)-\frac{\log n}{2\log\rho}\right|
    \leq \frac{c\log{\log{n}}}{\log{\rho}}.
\end{align*}
\item
Suppose both Assumptions~\ref{ass: 2} and~\ref{ass: 3} hold.
Then for any $c>1/2$ and $\mu$-almost every $\omega\in\Omega$, there exists $n_{c,\omega}\in\mathbb{N}$ such that for any $n\in\mathbb{N}$, if $n\geq n_{c,\omega}$ then:
\begin{align*}
    \left|R_n(\omega)-\frac{\log n}{2\log\rho}\right|\leq \frac{c\log{\log{n}}}{\log{\rho}}.
\end{align*}
\end{enumerate}
\end{thm}

Theorem~\ref{thm: 2} establishes the quantitative error bounds for the longest runs of partial quotients, strengthening the first-order ratio laws of Song--Zhou in~\eqref{eq:first-order-L} and Wang and Wu in~\eqref{eq:first-order-R} to an eventual two-sided additive refinement with explicit double-logarithmic additive error terms in both the fixed-value and maximised longest-run bounds.

\begin{thm}\label{thm: 2}
For any $c>1/2$ and almost every $x\in[0,1)\setminus\mathbb{Q}$, there exists $n_{c,x}\in\mathbb N$ such that for any $n\in\mathbb{N}$, if $n\geq n_{c,x}$ then:
\begin{align}\label{eq: main-L}
    \left|L_n(x,\lambda)-\frac{\log{n}}{2\log{\tau}}\right|
    \leq \frac{c\log{\log{n}}}{\log\tau},
\end{align}
where \(\tau=\tau(\lambda)\) is given in~\eqref{eq:first-order-L}, and:
\begin{align}\label{eq: main-R}
    \left|R_n(x)-\frac{\log{n}}{2\log{\varphi}}\right|
    \leq \frac{c\log{\log{n}}}{\log\varphi};
\end{align}
where $\varphi$ is given in~\eqref{eq:first-order-R}.
\end{thm}

In Theorem~\ref{thm: 2}, ``almost every'' is with respect to the Gauss measure. Since the Gauss measure and Lebesgue measure are mutually absolutely continuous on the unit interval, they have the same null sets; hence, the statement holds Lebesgue-almost everywhere as well.

The results are obtained by verifying Assumptions~\ref{ass: 1},~\ref{ass: 2}, and~\ref{ass: 3} for the partial-quotient process under the Gauss measure and then applying Theorem~\ref{thm: 1}. Concretely, quantitative mixing together with the two-sided exponential estimate for the constant-word event of a prescribed digit yields~\eqref{eq: main-L}. Moreover, the corresponding summed exponential bound over all digits also holds under the Gauss measure; hence, the maximised bound~\eqref{eq: main-R} follows.

\section{Proof of Theorem~\ref{thm: 1}}

\begin{proof}
Define, for any $n\in\mathbb{N}$ and $\rho>1$, 
\begin{align*}
    b_{n,\rho}\coloneqq \frac{\log n}{2\log\rho}.
\end{align*}
Pick any $c>1/2$ and $c_1\in(1/2,c)$. Define, for any $j\in\mathbb{N}$,
\begin{align*}
    n_j \coloneqq 2^j,
    &&
    k^-_j \coloneqq \left \lfloor b_{n_j,\rho}-\frac{c_1\log{j}}{\log{\rho}}\right\rfloor,
    &&
    k^+_j \coloneqq \left \lceil b_{n_j,\rho}+\frac{c_1\log{j}}{\log{\rho}}\right\rceil.
\end{align*}
Note that for any $j\in\mathbb{N}$, 
\begin{align}\label{eq:rho-bn}
    \rho^{-2b_{n_j}(\rho)}
    =\frac{1}{n_j}.
\end{align}

One establishes the upper bounds by applying the estimate assumptions and the Borel--Cantelli Lemma. Pick any $\omega\in\Omega$ and $m_*\in\mathcal{A}$. For any $j\in\mathbb{N}$, if $L_{n_j}(\omega,m_*)\geq k^+_j$ then there exists $t\in\{0,1,...n_j-k^+_j\}$ such that $T^t\omega\in\Delta_{k^+_j}{\left({m_*}\right)}$. By a union bound and $T$-invariance,
\begin{align*}
    \mu{\left(\left\{\omega\in\Omega:L_{n_j}(\omega,m_*)\geq k^+_j\right\}\right)}
    \leq \sum_{t=0}^{n_j-k^+_j}\mu{\left (T^{-t}\Delta_{k^+_j}{\left({m_*}\right)}\right)}
    \leq n_j\,\mu{\left(\Delta_{k^+_j}{\left({m_*}\right)}\right)}.
\end{align*}
By~\eqref{eq: ass 2} and~\eqref{eq:rho-bn}, for any $j\in\mathbb{N}$,
\begin{align*}
    n_j\,\mu{\left(\Delta_{k^+_j}{\left({m_*}\right)}\right)}
    \leq c_+\,n_j\,\rho^{-2k^+_j}
    \leq c_+\,n_j\,\rho^{-2b_{n_j,\rho}}\rho^{-2c_1\log{j}/\log{\rho}}
    = c_+\,j^{-2c_1}.
\end{align*}
Since $2c_1>1$ and $\sum_{j\in\mathbb{N}}j^{-2c_1}<+\infty$, the Borel--Cantelli Lemma implies that for $\mu$-almost every $\omega\in\Omega$, there exists $j_\omega\in\mathbb{N}$ such that for any $j\in\mathbb{N}$, if $j\geq j_\omega$ then:
\begin{align*}
     L_{n_j}(\omega,m_*)<k^+_j.
\end{align*}
For any $j\in\mathbb{N}$, if $R_{n_j}(\omega)\geq k^+_j$, then there exist $m\in\mathcal{A}$ and $t\in\{0,1,...n_j-k^+_j\}$ such that $T^t\omega\in\Delta_{k^+_j}{\left({m}\right)}$. By a union bound, the $T$-invariance,~\eqref{eq: ass 3} and~\eqref{eq:rho-bn}, for any $j\in\mathbb{N}$,
\begin{align*}
    \mu{\left(\left\{\omega\in\Omega:R_{n_j}(\omega)\geq k^+_j\right\}\right)}
    \leq n_j\sum_{m\in\mathcal{A}}\mu{\left(\Delta_{k^+_j}{\left(m\right)}\right)}
    \leq C_1\,n_j\,\rho^{-2k^+_j}
    \leq j^{-2c_1},
\end{align*}
Since $2c_1>1$ and $\sum_{j\in\mathbb{N}}j^{-2c_1}<+\infty$, the Borel--Cantelli Lemma implies that for $\mu$-almost every $\omega\in\Omega$, there exists $j_\omega\in\mathbb{N}$ such that for any $j\in\mathbb{N}$, if $j\geq j_\omega$ then:
\begin{align*}
     R_{n_j}(\omega)<k^+_j.
\end{align*}
For any $n\in\mathbb{N}$ and $j=\lfloor \log{n}/\log{2} \rfloor$, one obtains:
\begin{align*}
    b_{2^{j+1},\rho}=\frac{(j+1)\log2}{2\log\rho}
    \leq b_{n,\rho}+\frac{\log{2}}{2\log{\rho}},
    &&
    \frac{\log{(j+1)}}{\log{\rho}}\leq\frac{\log{\log{n}}}{\log{\rho}}+\frac{\log{2}-\log{\log{2}}}{\log{\rho}}.
\end{align*}
Therefore, for $\mu$-almost every $\omega\in\Omega$, there exists $j_\omega\in\mathbb{N}$ such that for any $n\in\mathbb{N}$, if $\lfloor \log{n}/\log{2} \rfloor\geq j_\omega$ then for the fixed-symbol longest-run statistic:
\begin{align}\label{eq: abstract-upper-L}
    L_n(\omega,m_*)
    &\leq L_{2^{j+1}}(\omega,m_*) <k^+_{j+1} \nonumber \\
    &\leq b_{n,\rho}+\frac{c_1\log{\log{n}}}{\log{\rho}}+1+\frac{\log{2}}{2\log{\rho}}+\frac{(\log{2}-\log{\log{2}})c_1}{\log{\rho}},
\end{align}
and for the maximised longest-run statistic:
\begin{align}\label{eq: abstract-upper-R}
    R_n(\omega)
    &\leq R_{2^{j+1}}(\omega) <k^+_{j+1} \nonumber \\
    &\leq b_{n,\rho}+\frac{c_1\log{\log{n}}}{\log{\rho}}+1+\frac{\log{2}}{2\log{\rho}}+\frac{(\log{2}-\log{\log{2}})c_1}{\log{\rho}}.
\end{align}

One establishes the lower bounds by applying separated trials, the mixing assumption, and the Borel--Cantelli Lemma.
By~\eqref{eq: ass 2} and~\eqref{eq:rho-bn}, for any $j\in\mathbb{N}$,
\begin{align*}
    p_j\coloneqq \mu{\left(\Delta_{k^-_j}{\left({m_*}\right)}\right)}\geq \frac{c_-}{n_j}j^{2c_1}.
\end{align*}
Define the gaps $(g_j)_{j\in\mathbb{N}}$ by, for any $j\in\mathbb{N}$,
\begin{align*}
    g_j \coloneqq \left\lceil -\frac{\log n_j}{\log{\theta}}\right\rceil,
\end{align*}
Note that for any $j\in\mathbb{N}$, if $j\geq\lceil{1+\log{C_0}/\log{2}\rceil}$ then:
\begin{align}\label{eq:gap-choice}
    C_0\,\theta^{g_j}\leq \frac12.
\end{align}
Define, for any $j\in\mathbb{N}$, $\ell_j\coloneqq k^-_j+g_j$, $M_j\coloneqq \lfloor{n_j}/{\ell_j}\rfloor$, and for any $r\in\{0,1,\dots,M_j-1\}$, $t_r\coloneqq r\ell_j$ and:
\begin{align*}
    B_{r,j}
    \coloneqq
    \left\{\omega\in\Omega:X_{t_r+1}(\omega)=\cdots=X_{t_r+k^-_j}(\omega)=m_*\right\}
    =T^{-t_r}\Delta_{k^-_j}{\left({m_*}\right)},
    &&
    N_j
    \coloneqq
    \sum_{r=0}^{M_j-1}\mathbf 1_{B_{r,j}}.
\end{align*}
Note that for any $j\in\mathbb{N}$,
\begin{align*}
    M_j
    \geq\frac{n_j}{\ell_j+1}
    \geq\left(\frac{1}{2\log{\rho}}+\frac{2}{\log{2}}-\frac{1}{\log{\theta}}\right)^{-1}\frac{n_j}{j\log{2}}.
\end{align*}
Since $\mu$ is $T$-invariant, one obtains from the definition of $(N_j)_{j\in\mathbb{N}}$ that for any $j\in\mathbb{N}$,
\begin{align}
\label{eq: ENj}
    \mathbb{E}[N_j]=M_jp_j
    \geq \left(\frac{1}{2\log{\rho}}+\frac{2}{\log{2}}-\frac{1}{\log{\theta}}\right)^{-1}\frac{c_-}{\log{2}}j^{2c_1-1}.
\end{align}
Define, for any $j\in\mathbb{N}$ and $r\in\{0,1,\dots,M_j-1\}$,
\begin{align*}
    A_{r,j}\coloneqq \Omega\setminus \bigcup_{s=0}^{r-1}B_{s,j}.
\end{align*}
By construction, for any $j\in\mathbb{N}$ and $r\in\{0,1,\dots,M_j-1\}$, one obtains two $\sigma$-algebras, namely $A_{r,j}\in\sigma(X_1,\dots,X_{t_r-g_j})$ and $B_{r,j}\in\sigma(X_{t_r+1},X_{t_r+2},\dots)$; while they are separated by a gap $g_j$. By applying~\eqref{eq: ass 1}, one obtains for any $j\in\mathbb{N}$ and $r\in\{0,1,\dots,M_j-1\}$,
\begin{align*}
    \mu{\left(A_{r,j}\cap B_{r,j}\right)}
    \geq \left(1-C_0\theta^{g_j}\right)\,\mu{(A_{r,j})}\,\mu{(B_{r,j})}
    =\left(1-C_0\theta^{g_j}\right)\,\mu{(A_{r,j})}\,p_j.
\end{align*}
By~\eqref{eq:gap-choice}, for any $j\in\mathbb{N}$ and $r\in\{0,1,\dots,M_j-1\}$, if $j\geq\lceil{1+\log{C_0}/\log{2}\rceil}$ then:
\begin{align*}
    \frac{\mu{\left(A_{r,j}\cap B_{r,j}\right)}}{\mu{\left(A_{r,j}\right)}}
    \geq \left(1-C_0\theta^{g_j}\right)p_j\geq \frac{p_j}{2};
\end{align*}
hence, 
\begin{align*}
    \mu{\left(A_{r+1,j}\right)}
    =\mu{\left(A_{r,j}\setminus B_{r,j}\right)}
    =\mu{\left(A_{r,j}\right)}\left(1-\frac{\mu{\left(A_{r,j}\cap B_{r,j}\right)}}{\mu{\left(A_{r,j}\right)}}\right)
    \leq \mu{\left(A_{r,j}\right)}\left(1-\frac{p_j}{2}\right).
\end{align*}
Note that for any $j\in\mathbb{N}$, $A_{0,j}=\Omega$, $A_{M_j,j}={N_j}^{-1}(\{0\})$, and:
\begin{align*}
    \mu{\left({N_j}^{-1}(\{0\})\right)}
    =\mu{\left(A_{M_j,j}\right)}
    \leq \left(1-\frac{p_j}{2}\right)^{M_j}
    \leq \exp{-\frac{M_jp_j}{2}}
    = \exp{-\frac{\mathbb{E}[N_j]}{2}}
\end{align*}
By~\eqref{eq: ENj} and $2c_1-1>0$, one obtains $\lim_{j\to+\infty}\mathbb{E}[N_j]/j^{2c_1-1}>0$ and $\sum_{j\in\mathbb{N}} \mu{({N_j}^{-1}(\{0\}))}<+\infty$. By the Borel--Cantelli Lemma, for $\mu$-almost every $\omega\in\Omega$, there exists $j_\omega\in\mathbb{N}$ such that for any $j\in\mathbb{N}$, if $j\geq j_\omega$ then $N_j(\omega)\geq 1$ and equivalently: 
\begin{align*}
    L_{n_j}(\omega,m_*)\geq k^-_j.
\end{align*}
For any $n\in\mathbb{N}$ and $j=\lfloor \log{n}/\log{2} \rfloor$, one obtains: $L_n(m_*)\geq L_{2^j}(m_*)\geq k^-_j$. For $\mu$-almost every $\omega\in\Omega$ and $n\in\mathbb{N}$, if $\lfloor \log{n}/\log{2} \rfloor\geq\max{\{j_\omega,\lceil{1+\log{C_0}/\log{2}\rceil}\}}$ then:
\begin{align}\label{eq:abstract-lower}
    R_n(\omega)
    \geq L_n(\omega,m_*)
    \geq b_{n,\rho}-\frac{c_1\log{\log{n}}}{\log{\rho}}-\left(1+\frac{\log{2}}{2\log{\rho}}-\frac{c_1\log{\log{2}}}{\log{\rho}}\right).
\end{align}

By combining~\eqref{eq: abstract-upper-L} and~\eqref{eq:abstract-lower}, one obtains that for $\mu$-almost every $\omega\in\Omega$, there exist $C_\omega>0$ and $n_\omega\in\mathbb{N}$ such that for any $n\in\mathbb{N}$, if $n\geq n_\omega$ then:
\begin{align*}
    \left |L_n(\omega,m_*)-b_{n,\rho}\right|
    \leq \frac{c_1\log{\log{n}}}{\log{\rho}} + C_\omega.
\end{align*}
Hence, \eqref{eq: main-L} is obtained for the fixed-symbol longest-run statistic by taking:
\begin{align*}
    n_{c,\omega}\coloneqq\max{\left\{n_\omega,
    \exp{\exp{\frac{C_\omega\log{\rho}}{c-c_1}}}\right\}},
\end{align*}
By combining~\eqref{eq: abstract-upper-R} and~\eqref{eq:abstract-lower},~\eqref{eq: main-R} is obtained for the maximised longest-run statistic.
\end{proof}

\section{Proof of Theorem~\ref{thm: 2}}

Theorem~\ref{thm: 2} is obtained by verifying the cylinder-measure inputs required by the general template in Theorem~\ref{thm: 1}. Assumption~\ref{ass: 1} is taken as given for the Gauss system. The remaining task is to verify Assumptions~\ref{ass: 2} and~\ref{ass: 3} via exponential estimates for constant-word cylinders, both for a fixed digit value and after summing over all digit values.

Let $\Omega=[0,1)\setminus\mathbb{Q}$ and $\mathcal{F}$ be the Borel $\sigma$-algebra relative to $\Omega$. Let $T:\Omega\to\Omega$ be the Gauss map on $\Omega$; that is for any $x\in\Omega$,
\begin{align*}
    T(x)\coloneqq\frac{1}{x}-\left\lfloor\frac{1}{x}\right\rfloor.
\end{align*}
Let $\mu:\mathcal{F}\to[0,1]$ be the Gauss measure on $(\Omega,\mathcal{F})$; that is for any $A\in\mathcal{F}$,
\begin{align}\label{eq: 16}
    \mu{(A)}\coloneqq\frac{1}{\log 2}\int_A\frac{\mathrm{d}x}{1+x},
\end{align}
The Gauss measure $\mu$ is invariant under the gauss map$T$ \cite[\S 15]{Khinchin1964ContinuedFractions}. The partial-quotient process $(a_n)_{n\in\mathbb{N}}$ on $\Omega$ is a dynamical process; that is, for any $n\in\mathbb{N}$ and $x\in\Omega$, 
\begin{align*}
    a_n(x)
    \coloneqq \left\lfloor \frac{1}{T^{\,n-1}(x)}\right\rfloor.
\end{align*}

Lemma~\ref{lem: assumption 1} below verifies Assumption~\ref{ass: 1}. 
This is the only place where the dependence structure of the partial quotients is considered.
\begin{lemma}\label{lem: assumption 1}
Let $\mu$ be the Gauss measure on $\Omega$, and for any $n\in\mathbb{N}$, $a_n:\Omega\to\mathbb{N}$ be the $n$-th partial quotient.
There exist $C_0>0$ and $0<\theta<1$ such that for any $r,g\in\mathbb{N}$ and events:
\begin{align*}
    A\in\mathcal{F}_{1,r}\coloneqq\sigma(a_1,\dots,a_r), && B\in\mathcal{F}_{r+g+1,+\infty}\coloneqq\sigma(a_{r+g+1},a_{r+g+2},\dots),
\end{align*}
one has:
\begin{align*}
    |\mu(A\cap B)-\mu(A)\mu(B)|
    \leq C_0\,\theta^{g}\,\mu(A)\mu(B).
\end{align*}
\end{lemma}
\begin{proof}
Let $\xi\coloneqq (I_k)_{k\in\mathbb{N}}$ be the standard partial quotient partition of $\Omega$ on the first level, that is for any $k\in\mathbb{N}$, $I_k\coloneqq\{x\in\Omega:a_1(x)=k\}$. Define the $\psi$-mixing coefficient $\psi_\xi:\mathbb{N}\to\mathbb{R}^+$ with respect to $\xi$ by, for any $g\in\mathbb{N}$,
\begin{align*}
    \psi_\xi{(g)}
    \coloneqq
    \sup_{r\in\mathbb{N}}\ 
    \sup_{\substack{A\in\mathcal F_{1,r},\ B\in\mathcal F_{r+g+1,+\infty}\\ \mu(A)\mu(B)>0}}
    \left|\frac{\mu(A\cap B)}{\mu(A)\mu(B)}-1\right|.
\end{align*}
By definition, for any $r\in\mathbb{N}$, $A\in\mathcal{F}_{1,r}$ and $B\in\mathcal{F}_{r+g+1,+\infty}$, if $\mu(A)\mu(B)>0$ then:
\begin{align}\label{eq:psi-mixing-assump1}
    \left|\frac{\mu(A\cap B)}{\mu(A)\mu(B)}-1\right| \leq \psi_\xi{(g)}.
\end{align}
By \cite[\S 4]{SchindlerZweimueller2023PrimeDigitsCF}, the Gauss system $(\Omega,\mathcal{F},\mu,T)$ is exponentially \(\psi\)-mixing with respect to $\xi$. Hence, there exist $C_0>0$ and $0<\theta<1$ such that for any $g\in\mathbb{N}$,
\begin{align}
\label{eq: CF}
    \psi_\xi(g)\le C_0\,\theta^{g}.
\end{align}
By combining~\eqref{eq:psi-mixing-assump1} and~\eqref{eq: CF}, Lemma~\ref{lem: assumption 1} is proved.
\end{proof}

Define, for any finite word $(d_1,\dots,d_k)\in\mathbb{N}^k$, the rank-$k$ cylinder set:
\begin{align*}
    \Delta{(d_1,\dots,d_k)}
    \coloneqq \bigcap_{r=1}^k \left\{x\in\Omega: a_r(x)=d_r\right\};
\end{align*}
and for any $\lambda\in\mathbb{N}$, a constant word set:
\begin{align*}
    \Delta_k{\left({\lambda}\right)}
    \coloneqq \bigcap_{r=1}^k \left\{x\in\Omega: a_r(x)=\lambda\right\}.
\end{align*}

A standard continued-fraction computation gives the diameter of an arbitrary cylinder in terms of convergents. Define, for any $k\in\mathbb{N}$ and $(d_1,\dots,d_k)\in\mathbb{N}^k$, $q_k$ be the denominator of the $k$-th convergent of $[d_1,\dots,d_k]$. It is well-known that for any $k\in\mathbb{N}$ and $(d_1,\dots,d_k)\in\mathbb{N}^k$,
\begin{align*}
    q_{-1}=0,\quad q_0=1,\quad q_k=d_k q_{k-1}+q_{k-2},
\end{align*}
and
\begin{align}\label{eq: 17}
    \operatorname{diam}{\Delta{(d_1,\dots,d_k)}}
    =\frac{1}{q_k(q_k+q_{k-1})}.
\end{align}
Let $\varphi\coloneqq(1+\sqrt5)/2$ be the golden ratio. Define $\tau:\mathbb{N}\to[\varphi,+\infty)$ by, for any  $\lambda\in\mathbb {N}$,
\begin{align*}
    \tau\coloneqq\tau(\lambda)\coloneqq \frac{\lambda+\sqrt{\lambda^2+4}}{2}.
\end{align*}

Lemma~\ref{lem: assumption 2} below verifies Assumption~\ref{ass: 2}. Therefore, with both Lemmas~\ref{lem: assumption 1} and~\ref{lem: assumption 2}, the first part of Theorem~\ref{thm: 1} applies and the proof for~\eqref{eq: main-L} in Theorem~\ref{thm: 2} is completed.
\begin{lemma}\label{lem: assumption 2}
For any $\lambda\in\mathbb{N}$, there exist $c_-,c_+>0$ such that for any $k\in\mathbb{N}$,
\begin{align*}
    c_-\,\tau^{-2k}\leq \mu{(\Delta_k{(\lambda)})}\leq c_+\,\tau^{-2k}.
\end{align*}
\end{lemma}
\begin{proof}
Pick any $k\in\mathbb{N}$. The denominator $q_k$ of the $k$-th convergent of $[\lambda,\lambda,\ldots]$ satisfies:
\begin{align*}
    \frac{\lambda}{\sqrt{\lambda^2+4}}\tau^{k}\leq q_{k}=\frac{\tau^{k+1}-(-\tau)^{-(k+1)}}{\sqrt{\lambda^2+4}}\leq \tau^{k}.
\end{align*}
By~\eqref{eq: 17},
\begin{align*}
    \frac{\tau}{\tau+1} \tau^{-2k}
    \leq \operatorname{diam}{\Delta_k{(\lambda)}}
    =\frac{1}{q_k(q_k+q_{k-1})}
    \leq\frac{(\lambda^2+4)\tau}{(\tau+1)\lambda^2}\tau^{-2k}
\end{align*}
By~\eqref{eq: 16},
\begin{align*}
     \frac{1}{2\log{2}}\frac{\tau}{\tau+1}\tau^{-2k}\leq \mu{(\Delta_k{(\lambda)})}\leq \frac{(\lambda^2+4)\tau}{(\tau+1)\lambda^2\log{2}}\tau^{-2k}.
\end{align*}
Lemma~\ref{lem: assumption 2} is proved.
\end{proof}

Lemma~\ref{lem: assumption 3} below verifies Assumption~\ref{ass: 3}. Therefore, with Lemmas~\ref{lem: assumption 1},~\ref{lem: assumption 2}, and~\ref{lem: assumption 3}, the second part of Theorem~\ref{thm: 1} applies and the proof for~\eqref{eq: main-R} in Theorem~\ref{thm: 2} is completed.
\begin{lemma}\label{lem: assumption 3}
For any $k\in\mathbb{N}$,
\begin{align*}
    \sum_{\lambda\in\mathbb{N}}\mu{(\Delta_k{\left({\lambda}\right)})}\leq\varphi^{-2(k-1)}.
\end{align*}
\end{lemma}
\begin{proof}
For $k=1$, the estimate is trivial. By~\eqref{eq: 16} and~\eqref{eq: 17}, for any $k\in\mathbb{N}\setminus\{1\}$,
\begin{align*}
    \sum_{\lambda\in\mathbb{N}}\mu{(\Delta_k{\left({\lambda}\right)})}
    &= \frac{1}{\log{2}}\sum_{\lambda\in\mathbb{N}}\int_{\Delta_k{(\lambda)}}\frac{\mathrm{d}x}{1+x} \\
    &\leq \frac{1}{\log 2}\int_{\Delta_k{(1)}}\frac{\mathrm{d}x}{1+x}+\frac{1}{\log 2}\sum_{\lambda=2}^{+\infty}\operatorname{diam}{\Delta_k{(\lambda)}} \\
    &\leq \frac{2}{3\log 2}\int_{\Delta_k{(1)}}\mathrm{d}x+\frac{1}{\log 2}\sum_{\lambda=2}^{+\infty} \lambda^{-2k}\\
    &\leq \frac{2}{3\log 2}\operatorname{diam}{\Delta_k{(1)}}+\frac{4^{-(k-2)}}{\log 2}\sum_{\lambda=2}^{+\infty} \lambda^{-4}\\
    &\leq \frac{2\varphi^{-(2k-1)}}{3\log 2} +\frac{\varphi^{-2(k-2)}}{\log 2}\left(\frac{\pi^4}{90}-1\right) \\
    &= \frac{1}{\log{2}}\left(\frac{2}{3\varphi}+\left(\frac{\pi^4}{90}-1\right)\varphi^{2}\right)\varphi^{-2(k-1)}.
\end{align*}
Lemma~\ref{lem: assumption 3} is proved.
\end{proof}

\bibliographystyle{siam} 
\bibliography{name}

\end{document}